\documentclass{article}

\usepackage[utf8]{inputenc}

\usepackage{amsfonts,amsmath,amsthm}

\usepackage{todonotes}
\usepackage{mathtools}

 \usepackage{xcolor}
 
\usepackage{amssymb}

\usepackage{algorithm}
\usepackage{algorithmic}
\usepackage{todonotes}

\newtheorem{theorem}{Theorem}

\newtheorem{definition}{Definition}

\newtheorem{remark}[theorem]{Remark}
\usepackage{multirow}
 \newcommand{\E}{{\mathbb E}}
\def\R{{\mathbb R}}

\def\N{{\mathbb N}}

\newcommand{\mA}{\mathsf{A}}

\newcommand{\mPP}{\mathsf{\Phi}}
\usepackage[affil-it]{authblk}
\usepackage{hyperref}

\DeclareMathOperator{\supp}{supp}

\title{Greedy algorithms for learning via exponential-polynomial splines}

\author[1]{R. Campagna 
\thanks{\href{rosanna.campagna@unicampania.it }{rosanna.campagna@unicampania.it}}}
\author[2]{S. De Marchi 
\thanks{\href{mailto:demarchi@math.unipd.it}{demarchi@math.unipd.it}}}
\author[3]{E. Perracchione
\thanks{\href{perracchione@dima.unige.it}{perracchione@dima.unige.it}}}
\author[4]{G. Santin 
\thanks{\href{mailto:gsantin@fbk.eu}{gsantin@fbk.eu}}}

\affil[1]{University of Campania   ``L. Vanvitelli'', Italy}
\affil[2]{University of Padova, Italy}
\affil[3]{CNR-SPIN, University of Genova, Italy}
\affil[4]{DIGIS, Bruno Kessler Foundation, Italy}

\begin{document}
\maketitle
\begin{abstract}
Kernel-based schemes are  state-of-the-art techniques for learning by data. In this work we extend some ideas about kernel-based greedy algorithms to exponential-polynomial splines, whose main drawback consists in possible {\em overfitting} and consequent oscillations of the approximant. To partially overcome this issue, we introduce two algorithms which perform an adaptive selection of the spline interpolation points based on the minimization either of the sample residuals ({$f$-greedy}), or of an upper bound for the approximation error based on the  spline Lebesgue function ({$\lambda$-greedy}). Both methods allow us to obtain an adaptive selection of the sampling points, i.e. the spline nodes. However, while the {$f$-greedy} selection is tailored to one specific target function, the $\lambda$-greedy algorithm is independent of the function values and enables us to define a priori optimal interpolation nodes. 
\end{abstract}

\section{Introduction}

Scattered data interpolation is one of the most investigated topics in the field of numerical analysis, and it is successfully used for many applications. As a 
consequence, many methods have been developed, including interpolation with polynomials of total degree (see e.g. \cite{Bos2017}), splines approximation  
\cite{HP2021,CR2018,deBoor76}, with its recent developments in the context of machine learning \cite{Bohra,Unser}, and kernel-based methods (refer e.g. to 
\cite{Fasshauer15,Wendland2005}). 

More recently, the so-called Exponential-Polynomial Splines (EPS) have been introduced with the main purpose of approximating univariate multi-exponential decay functions with a smoothing effect \cite{bayonacc,Campagna2019,AMC_CCC}. Such a smoothing strategy is implemented by considering a regularization parameter. 

In this paper we are interested in the design of appropriate sampling strategies for EPS interpolation. Namely, assuming to be given either only an input space discretization, or a dataset of input points and corresponding function evaluations, we aim at selecting a small subset of approximation points to be used to construct the EPS interpolant.
We consider in particular incremental methods that, given an initial set of samples, construct an EPS interpolant by iteratively selecting a new point at at each iteration. The iterative rule is dictated by greedy methods (see \cite{Temlyakov2008}), which have been investigated e.g. for kernel methods (see e.g. \cite{Haasdonk2018,DSW,SH2017,WH2013})
and lead to sparse models which turn out to be helpful in many applications, see e.g. \cite{DUTTA2021110378}. This iterative selection is a convenient proxy for the optimal selection of the sampling points, which is in turn usually an extremely computationally demanding procedure.

Since the greedy selection is based on the maximization of an error indicator, to use such schemes in the context of EPS we first study pointwise error bounds. We are able to bound the pointwise error thanks to the definition of the cardinal form of the EPS interpolant that then allows us to introduce the Lebesgue function and constant \cite{Brutman,Brutman1}. The latter are known to be stability indicators for polynomial bases; see e.g.  \cite{Bayliss,BERRUT199777,BOS200615,BOS2011504,Bos2017,DEMARCHI2021125628}. Based on this new error indicator, we define an algorithm for selecting data-independent  points for EPS. 

Furthermore, we propose a second extraction strategy that takes into account also the function values. This kind of approach is usually more expensive, but it allows us to select points that are tailored to one specific target function, and thus are usually able to better resolve local features such as steep gradients or oscillations.

In both cases, we numerically explore the behavior of the node distribution for the spline basis, and we test our findings under different perspectives. 

The paper is organized as follows. In Section \ref{preliminari} we briefly review the basics of greedy methods and EPS interpolation. Error bounds and adaptive strategies for the greedy selection of the nodes are presented in Section \ref{sec:greedy}. Some numerical experiments are presented and discussed in Section \ref{sec:numerics}, while conclusions with an outline of future works are provided in Section \ref{Concl}.

\section{Exponential splines and greedy schemes}
\label{preliminari}

In this section we present the main features of EPS and greedy methods. 

We consider a function $f:[a,b]\to\R$ with $[a, b]\subset\R$ and an associated set of {function} values $F\coloneqq\{y_i\coloneqq f(x_i)\}_{i=1}^{n}$ sampled at a data set  $X\coloneqq\{x_i\}_{i=1}^{n}$, with $a=x_1 < x_2< \dots < x_{n} = b$. Our goal is to construct an approximation of the unknown function $f$, and we concentrate on interpolatory schemes, i.e., our model $I_X$ of $f$ satisfies $I_X(x_i) = y_i$, $1\leq i\leq n$.

In general terms, given a normed linear space of functions defined in $[a,b] \subset \Omega $,   and an associated basis   $\{b_j\}_{j=1}^n\subset C([a, b])$, an interpolant $I_X:[a,b]\to\R$ may be defined as 
\begin{equation*}
    I_{X}(x) = \sum_{j=1}^n c_j b_j(x), \quad x \in [a,b].
\end{equation*}
Provided that $\{b_j\}_{j=1}^n$ form a Haar system (see e.g. \cite{Wendland2005}), the matrix $\mA\in\R^{n\times n}$ with $\mA_{ij} \coloneqq b_j(x_i)$, $1\leq i,j\leq n$, is invertible for any set of interpolation points, and the coefficients $\boldsymbol{c}\coloneqq[c_1, \dots, c_n]^{\intercal}\in \R^{n}$ of the interpolant may be determined by solving the system
\begin{equation*}
    \mA \boldsymbol{c} = \boldsymbol{y},    
\end{equation*}
where $\boldsymbol{y}\coloneqq[y_1, \dots, y_n]^{\intercal}\in \R^{n}$. Popular basis functions that meet these requirements are, for instance, Radial Basis Functions (RBFs), or monomials of total degree $n-1$. 

\subsection{Exponential splines}\label{sec:exp_splines}
In this work we focus our attention on a particular
 basis $\{b_j\}_{j=1}^n$ of 
splines that was introduced in \cite{Campagna2019}. In this section we recall the  definition of the smoothing spline model introduced in the cited paper and give the basis definition in   details.

\begin{definition}
    Let $\{(x_i,y_i)\}_{i=1}^{n}\subset\Omega\times\R$ be given, with $a=x_1$ and $b=x_n$. 
Let $w_1,\ldots, w_n$ be non zero weights, $\beta>0$ be a regularization parameter. 

For $\alpha>0$ and $\mathcal{L}_2u:=u''+2\alpha\, u'+\alpha^2\, u$, we denote as $\mathcal{L}_2^*$ the adjoint, and as $\mathbb{E}_{4, \alpha} := \{e^{\alpha x},\,x e^{\alpha x},\ e^{-\alpha x},\  x e^{-\alpha x}\}$ the null space of $\mathcal{L}_2^*\mathcal{L}_2$.

Then the \textbf{smoothing exponential  spline} $I_{X,\alpha}(f):[a,b]\to\R$  is the solution of the penalized least square problem
\begin{equation}\label{defexpspline}
 \min_{c_1,\ldots, c_n}\ 
 \displaystyle {\sum}_{i=1}^n w_i \left(y_i-\sum_{j=1}^n c_j \varphi_j(x_i)\right)^2\, +\, \beta   \, \displaystyle\int_{a}^{b} \left( \sum_{j=1}^{n}c_j   {\cal L}_2\varphi_j(x)\right)^2 dx, 
\end{equation}
with $\varphi_{j}|_{[x_i, x_{i+1}]} \in \mathbb{E}_{4, \alpha}$, where $_{|A}$ denotes the restriction on a set $A\subset\R$. 
\end{definition}

In the same paper \cite{Campagna2019}, an  {\em optimal} basis $\{\varphi_j\}_{j=1}^n$ of \emph{exponential  B-splines}, also referred to as Generalized B-splines {GB-splines}, has been defined. The basis functions have the following properties (see fig. \ref{bspline}): are  {bell-shaped} with compact support, identified by  $5$ {nodes}, with the blending segments belonging to $\E_{4,\alpha}$
and $C^2$-smoothness. 
The generic basis function $\varphi$ can then be expressed as
   \begin{align}\label{eq:global_basis}
\varphi(x) |_{[x_i, x_{i+1}]} =
\sum_{k=1}^4 b_{i,k} g_{k}(x),\quad g_{k} \in \E_{4, \alpha},
\end{align}
where $x_i$ with $i=1,\dots, n-1$, denotes the left point of the partition element, and $k=1,\dots, 4$ denotes the index of the local basis element.
  Indeed,   each function $\varphi$ in the form \eqref{eq:global_basis} has $4 (n-1)$ degrees of freedom, given by the coefficients $\{b_{i,k}\}$ with $1\leq i\leq n-1$, $1\leq k\leq 4$. Then, to define such a  basis of dimension $n$, the nodes vector has to be  augmented with two extra nodes before $x_1$ and two others after $x_n$,   i.e.   an {augmented} node set as $x_{-1}<x_0<x_1=a< \ldots < x_n=b<x_{n+1}<x_{n+2}$ has to be considered.
Such extra nodes affect the construction of the so called {\em boundary} basis functions: $\varphi_1, \varphi_2$ and   $\varphi_{n-1}, \varphi_{n}$.

     \begin{figure}[htb]
\centering \includegraphics[width=0.5\textwidth]{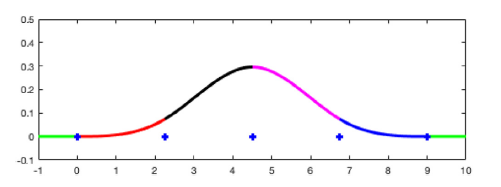}
\centering\caption{An example of GB-spline
with   segments in the spaces $\E_{4,\alpha}$. \label{bspline}}
\end{figure} 
 
\begin{remark}
We assume that these four points are fixed, indeed, numerically we observed that, for our scope, the approximation is not very sensitive with respect to their selection.
\end{remark}
In this paper we focus on unweighted interpolation, i.e., no smoothing parameter nor weights are considered: we set $\beta=0$ and $w_i=1$, $1\leq i\leq n$. Moreover, the reason why the authors in \cite{Campagna2019}
assume $\alpha>0$
is due to the fact that they are interested in modeling functions that decay exponentially. In order to force the model decreasing outside, the authors assume also that 
 the segments of the GB-splines falling outside $[x_1,x_n]$  are defined in a two dimensional space $\E_{2,\alpha} : =  \{ e^{-\alpha x},  x\ e^{-\alpha x}\}$, null space of $\mathcal{L}_2$
(for further details see \cite{Campagna2019}).  
The global space of the exponential-polynomial splines ${\mathcal E}_{X, \alpha}(\E_{2,\alpha},\E_{4,\alpha})$ is then defined by gluing local patches defined over each interval such that 
 ${\mathcal E}_{X, \alpha}(\E_{2,\alpha},\E_{4,\alpha})\subset C^2([a,b])$.

Here instead, we focus on a more general model, where only $\mathbb{E}_4$ is considered.
Then, we assume ${\mathcal E}_{X,\alpha}\coloneqq   {\mathcal E}_{X,\alpha}(\E_4)$ and any of its elements can be expressed using coefficients $b_{i,k}$ as in (\ref{eq:global_basis}).


 Following \cite{Campagna2019}, we use a Bernstein-like basis to represent  each segment of the GB-splines $\{\varphi_j\}_{j=1}^n$ (see \cite[Appendix]{Campagna2019} for an explicit construction), as follows:
\begin{equation}\label{ber}
 \varphi_j(x)|_{[x_i, x_{i+1}]}=\sum_{k=1}^4\gamma_{i,j,k}{B}_k(x-x_j),
\end{equation}
where ${B}_k$, $k=1,\ldots,4,$ are Bernstein like functions, $1\leq i\leq n-1$ and $1\leq j\leq n$. The existence and uniqueness of such a functional space is provided in \cite[Theorem 2.1]{Campagna2019}.  The advantage of the GB-spline basis is that the  computations can be performed locally in the support of each $\varphi_j$, indeed $\supp(\varphi_j)\subset[x_{j-2}, x_{j+2}],\;\; 1\leq j\leq n$. In particular, the global interpolation matrix $\mPP$ with entries given by $\mPP_{ij} \coloneqq \varphi_j(x_i)$ is nonzero only on the main diagonal and on the two upper and lower diagonals.\\

Fixed the space of the described exponential splines, any function $f:[a,b]\to\R$ can now be approximated on  $X\coloneqq\{x_i\}_{i=1}^{n}$ via  exponential-polynomial splines interpolation of the associated samples $F = \{y_i=f(x_i)\}_{i=1}^{n}$, i.e., 
\begin{align}\label{eq:interpolant}
I_{X,\alpha}(f)(x) = \sum_{j=1}^n c_j \varphi_j(x),
\end{align} 
with a vector of coefficients $\boldsymbol{c}\coloneqq[c_1, \dots, c_n]^{\intercal}\in \R^{n}$ such that 
\begin{equation}\label{sys}
    \mPP \boldsymbol{c} = \boldsymbol{y},    
\end{equation}
where $\boldsymbol{y}\coloneqq[y_1, \dots, y_n]^{\intercal}\in \R^{n}$. Observe that we have $I_{X,\alpha}(f) = f$ for all $f\in {\mathcal E}_{X,\alpha}$, i.e., every function in ${\mathcal E}_{X,\alpha}$ is uniquely determined by its values on $X$. We summarize in Algorithm 1 the steps for computing the EPS interpolant. 

\begin{algorithm}[H]
\small
\caption{\textbf{Pseudo-code for EPS}}
\begin{algorithmic}[1]
\STATE Definition of the function space $\mathbb{E}_4$   by setting $\alpha$.
\STATE Definition of the augmented nodes to define the boundary basis functions.
\STATE Definition of the GB-spline basis functions in   Bernstein-like basis, $\{\varphi_j\}_{j=1}^n$.\\
\STATE Computation of the collocation matrix as in \eqref{sys} and solution of the collocation system. 
 		\end{algorithmic}
	\label{alg:alg1}
\end{algorithm}

Before investigating the use of greedy schemes for EPS, we briefly recall the main ideas behind greedy techniques. 

\subsection{Greedy schemes}

Given $X$ and $F$, the main goal of the greedy algorithms consists in selecting a suitable subset $\tilde{X} \subset X$ so that the {\em greedy interpolant} is constructed on a reduced number of data producing an approximation of $I_X$. Such iterative algorithms belong essentially to two classes:
\begin{itemize}
    \item Residual-based greedy schemes: the set $\tilde{X}$ is constructed taking into account the function values $F$.
    \item Error-based greedy methods: the set $\tilde{X}$ is built independently of the function values $F$.
\end{itemize} 

The general iterative rules for these two algorithms are summarized in Table \ref{tab_it}, where $\lambda$ denotes a pointwise approximation error independent of the function values. Both methods will be investigated in the next section for the special case of EPS.

\begin{table}[!h]

\begin{tabular}{l|c}
 \hline 
 \hline \\ [-0.9ex] 
 Greedy Method & Iterative Rule \\ [0.9ex] 
 \hline  \\ [-0.9ex] 
 Residual-based   & $\displaystyle x^*={\rm argmax}_{x \in \Omega \setminus { X}} |f(x)-I_{{X},\alpha}(f)(x)|$
 \\   [1.9ex] 
 Error-based &   $\displaystyle x^*={\rm argmax}_{x \in \Omega \setminus { X}} \lambda(x)$  \\ [1.9ex] 
 \hline \hline
\end{tabular}
\caption{Iterative rules for residual and error-based greedy strategies.}
\label{tab_it}
\end{table}

\section{Greedy schemes for EPS}
\label{sec:greedy}
In this section we first recall a simple residual-based greedy scheme, that is known as $f$-greedy in  literature (see \cite{SchWen2000}), and that can be easily used with any approximation basis. On the other hand, error-based greedy schemes need to be tailored for the considered basis, and we will discuss their derivation in the case of EPS.

\subsection{Residual-based greedy selection}

As already mentioned, $f$-greedy schemes are quite straightforward to extend to any kind of basis. 
Precisely, we consider an initial (training) set of sorted data ${\tilde X}=\{x_i\}_{i=1}^q \subset X$, with $x_1=a, x_q=b$ and we also keep the augmented nodes fixed.
Then, given $F$ and a fixed tolerance $\tau$, the residual-based greedy scheme for exponential splines is summarized in Algorithm 2. 

\begin{algorithm}[H]
\small
\caption{\textbf{Pseudo-code for the $f$-greedy algorithm}}
\begin{algorithmic}[1]
\STATE Take an initial set of sorted data ${\tilde X}=\{x_i\}_{i=1}^q \subset X$; $x_1=a, x_q=b$ and $q\geq2$. 
\STATE Compute an initial interpolant $I_{\tilde{X},\alpha}(f)$ as in \eqref{eq:interpolant}.
\STATE While $\max_{x_i \in X\setminus {\tilde X}} |f(x_i)-I_{\tilde{X},\alpha}(f)(x_i)| > \tau $:
    \begin{enumerate}    
    \item Define $x^*={\rm argmax}_{x_i \in X\setminus {\tilde X}} |f(x_i)-I_{\tilde{X},\alpha}(f)(x_i)|$.
    \item Set ${\tilde X}= {\tilde X}\cup \{x^*\}$ and sort  ${\tilde X}$.
    \item Compute $I_{\tilde{X},\alpha}(f)$ as in \eqref{eq:interpolant}.
\end{enumerate}
 		\end{algorithmic}
	\label{alg:alg2}
\end{algorithm}

The result of the $f$-greedy scheme is thus a set of data locations ${\tilde X}=\{x_i\}_{i=1}^{\tilde{n}}$ with $a=x_1 < \ldots < x_{\tilde{n}} = b$, and the corresponding interpolant $I_{\tilde{X},\alpha}(f)$.
Since we usually have ${\tilde{n}} \ll n$, the greedy interpolant $I_{\tilde{X},\alpha}(f)$ can be understood as a sparse approximation of $I_{{X},\alpha}(f)$.

This scheme is very easy to implement, and additionally the interpolation points are selected adaptively in order to be suited for the particular target function $f$, and they are thus expected to provide an accurate approximation. 

On the other hand this adaptivity may backfire. Indeed, in many applications one aims instead at the selection of a set of interpolation points that can be used to approximate $N_f\in\N$ different functions. In this case, the $f$-greedy algorithm should be executed $N_f$ times, and this constitutes a computational drawback. To deal with this scenario, we drive our attention towards error-based greedy schemes.

\subsection{Error-based greedy selection}

To investigate the error-based greedy selection, we need to introduce a pointwise error bound for EPS interpolation.

\subsubsection{Lagrange functions and Lebesgue constant}
Given $\{\varphi_{j}\}_{j=1}^n$ as in \eqref{eq:global_basis}, since the associated matrix $\mPP$ is invertible we may write $d_{j\ell}\coloneqq(\mPP^{-1})_{j\ell}$. In this way we have that the functions
\begin{align*}
    \psi_\ell(x)\coloneqq \sum_{j=1}^n d_{j\ell}\varphi_j(x),\;\; 1\leq \ell\leq n,
\end{align*}
satisfy the cardinal conditions 
\begin{align}\label{eq:cardinal_conditions}
    \psi_\ell(x_i) = \delta_{i\ell} \;\; 1\leq i, \ell\leq n,
\end{align}
i.e., they are a global Lagrange (or cardinal) basis. To see this, just observe that for $1\leq i,\ell \leq n$, it holds true that
\begin{align*}
\psi_\ell(x_i)
&= \sum_{j=1}^n \varphi_j(x_i) d_{j\ell}
= \sum_{j=1}^n \mPP_{ij} (\mPP^{-1})_{j\ell}
= \left(\mPP\cdot \mPP^{-1}\right)_{i\ell}
=\delta_{i\ell},\quad 1\leq i, \ell\leq n.
\end{align*}
Using the cardinal basis, the interpolant \eqref{eq:interpolant} may be written as
\begin{align}\label{eq:lagrange_interpolant}
    I_{X,\alpha}(f)(x) = \sum_{j=1}^n f(x_j) \psi_j(x),\;\; x\in [a, b].
\end{align}

Some examples of cardinal bases for the EPS are plotted in Figure \ref{fig:1}. In this illustrative example, the cardinal functions are computed for $n=8$ equispaced, Halton and Chebyshev data locations, and are evaluated on $400$ equispaced points.

\begin{figure}
    \centering
   \includegraphics[clip,trim=2cm 0.73cm 2cm 0.73cm,scale = 0.35]{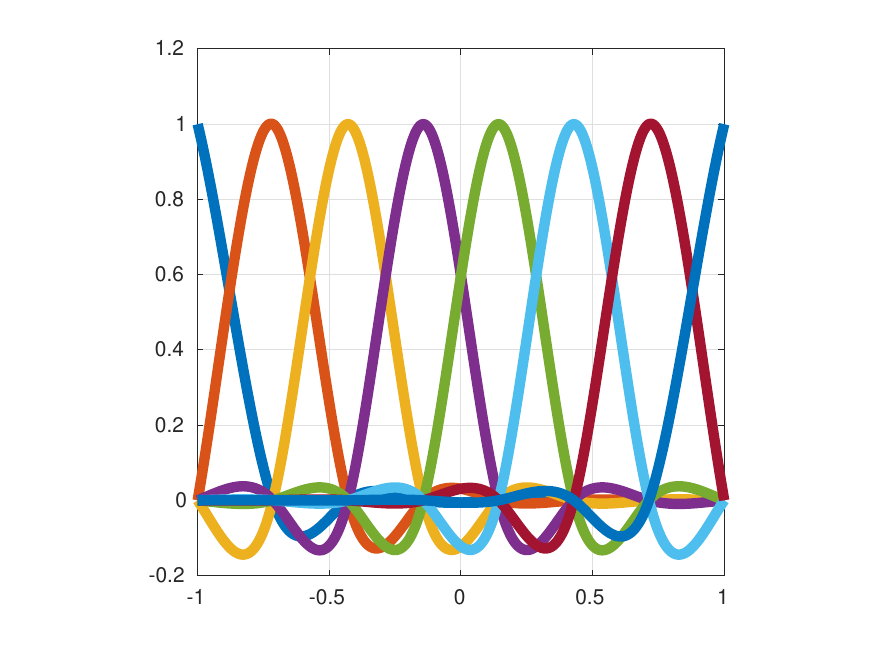}
    \includegraphics[clip,trim=2cm 0.73cm 2cm 0.73cm,scale = 0.35]{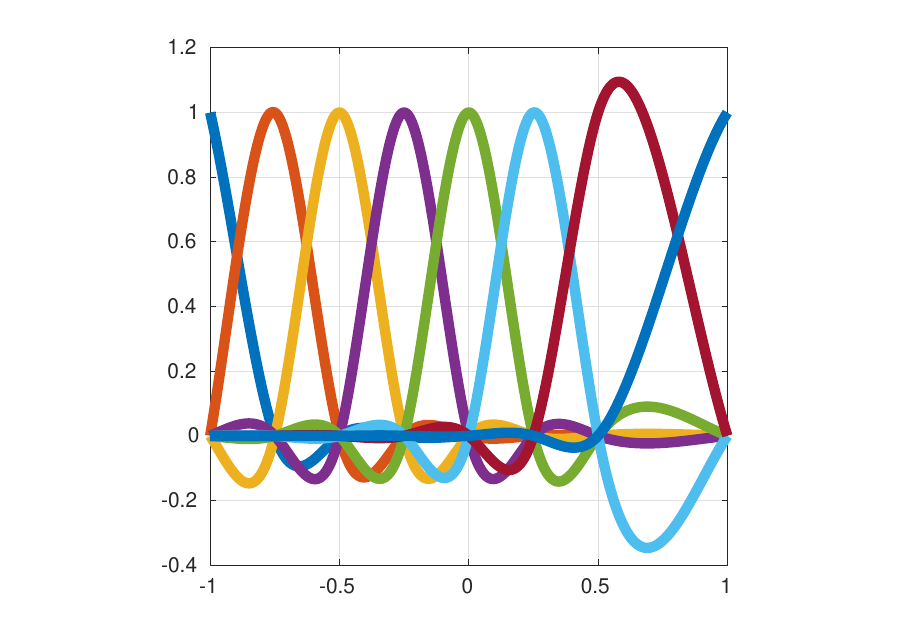}
    \includegraphics[clip,trim=2cm 0.73cm 2cm 0.73cm,scale = 0.35]{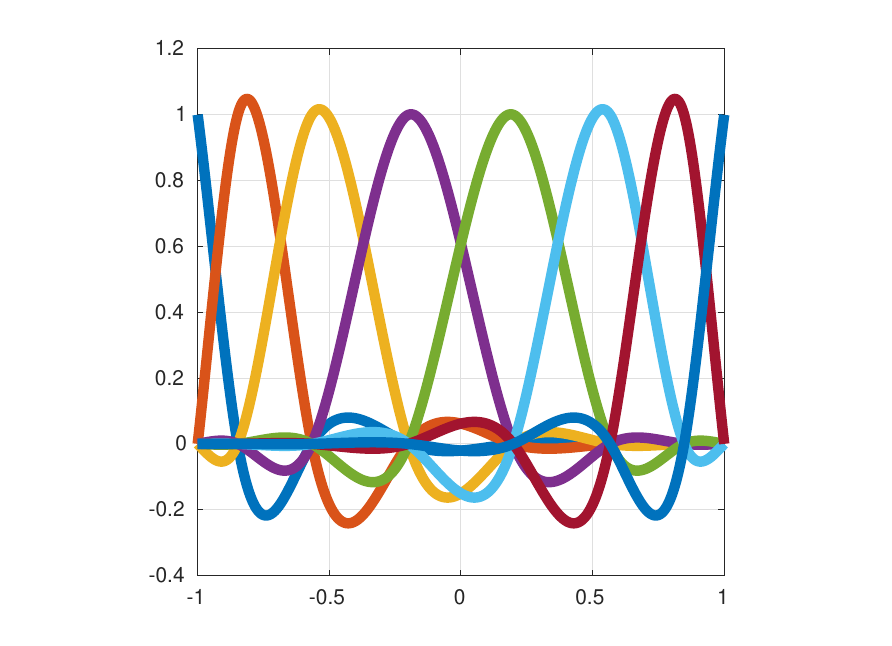}  
    \caption{From left to right: cardinal functions computed on $n=8$ equispaced, Halton and Chebyshev data, respectively.}
    \label{fig:1}
\end{figure}

Once the cardinal basis is computed,  the Lebesgue function is defined in the usual way as 
\begin{equation*}
    \lambda(x)\coloneqq\lambda(x_1,\ldots,x_n;x)\coloneqq\sum_{j=1}^n |\psi_j(x)|,\;\;x\in [a,b],
\end{equation*}
and its maximum value is called the Lebesgue constant, defined by
\begin{equation*}
\Lambda\coloneqq\Lambda(x_1,\ldots,x_n;x)\coloneqq\textrm{sup}_{a\leq x\leq b}\sum_{j=1}^n |\psi_j(x)|.
\end{equation*}
Both $\lambda$ and $\Lambda$ depend on the location of the interpolation points and on their number $n$, but not on the function values and, as will be evident, they are  stability indicators. 

\subsubsection{Lagrange functions and stability}

It is known \cite{novakovic2010estimates} that, in the interpolation problem, the sensitivity of the solution is   determined by the condition number $\kappa_p (\mPP)$ of the collocation matrix  
\begin{equation}\label{condPhi}
\kappa_p (\mPP)=\|\mPP\|_p \|\mPP^{-1} \|_p,
\end{equation}
where $\|\cdot \|_p$   denotes a standard operator p-norm, with $1 \leq  p \leq  \infty $.
Carl de Boor \cite{deBoor76} conjectured that the interpolation by (polynomial) B-splines of degree $d$ at node averages is bounded by a function that depends only on $d$, regardless of the nodes themselves. Several works  disproved this conjecture; moreover improvements are also available in literature (see \cite{LYCHE1978202,novakovic2010estimates,SCHERER1999217}). In the numerical experiments we will use (\ref{condPhi}), with $p=2$, to estimate the amplification  factor for the data noise.

In this work we aim to define a theoretical estimate for the residual bound, to formulate   
a possible stopping rule based {\em only} on the nodes distributions. At this end we relate the Lebesgue constant $\Lambda$ and the condition number $\kappa_2 (\mPP)$.

If the interpolant is expressed in the Lagrange basis, then the Lebesgue constant also estimates the conditioning of the interpolation problem.
Following \cite{Carnicer:78093} we define the evaluation functionals
 $$\xi_i(f)\coloneqq f(x_i)=y_i,\quad i=1,\ldots,n. $$ Then, we remark that (\cite[\S 2]{cheney2009course}) the conditioning of the Lagrange representation corresponding to the evaluation functionals $\xi_1,\ldots ,\xi_n$ coincides with the Lebesgue function, i.e. 
\begin{equation}
   \label{mucsi} 
\mu(x ; \xi) = \sum_{i=1}^n | \psi_i (x) | = \lambda(x).
\end{equation}
Moreover, the following theorem  shows that the Lagrange representation, in terms of the evaluation functionals, has optimal conditioning. The conditioning of any other representation is greater than the conditioning of the Lagrange representation and the quotient can be bounded by the {\em Skeel condition number} of  the inverse   collocation matrix of the corresponding basis \cite{Carnicer:78093}. In a similar way,  it also holds true for $I_{X,\alpha}(f)$, being expressed in cardinal form as in (\ref{eq:lagrange_interpolant}).  
\begin{theorem}
 Let $x \in [x,b]$,  $I_{X,\alpha}(f)(x) = \sum_{i=1}^n c_i \phi_i(x)$ be a representation of the Lagrange interpolation operator 
  and let $\mPP = (\phi_j (x_i ))$,  $i, j=1,\ldots,n$, be the collocation matrix. Then the following inequality holds true
\[
\mu(x ; \xi) \leq \mu(x ; \phi) \leq \mu_{\textrm{Skeel}}(\mPP^{-1})\mu(x ; \xi),
\]
where $\mu(x ; \xi)$ is defined in (\ref{mucsi}).
\end{theorem}
The conditioning  
measures   the sensitivity of the representation to the error propagation.
As concern this, let us consider the problem of interpolating perturbed data, i.e., for $i=1,\ldots,n$: $${\tilde{y}_i}={ {y}_i}(1+\sigma_i) \quad \textrm{for some } \quad \boldsymbol{\sigma }=(\sigma_1,\ldots,\sigma_n), \quad s.t.\quad  | { \sigma_i } |\leq C\varepsilon+O(\varepsilon^2),$$
where $C$ is a constant and $\varepsilon$ denotes the {\em machine precision}.   Given $x \in [a,b]$, let us define   $$\eta(\boldsymbol{f},x)=\sum_{j=1}^n | \psi_j(x) | \|\boldsymbol{f }\|_{\infty},$$
then the so-called rounding error can be bounded by 
\[\varepsilon_{\textrm{rounding}}(x)=|{I}_{X,\alpha}(\tilde{f})-I_{X,\alpha} (f) (x)|\leq  \|{\bf \tilde{f}}-{\bf {f}}\|_{\infty}\eta({\bf f},x).  \]
The rounding error, together with the approximation error
$$\varepsilon_{\textrm{approx}}(x)=\left|f(x) - I_{X,\alpha}(f)(x)\right|,$$
defines the global error upper bound as 
$$
\varepsilon(x) = \varepsilon_{\textrm{approx}}(x) + \varepsilon_{\textrm{rounding}}(x), \quad x \in [a,b].
$$

To define our greedy strategy, we essentially have to bound the approximation error, because differently from the rounding error, it is independent of the function values. To this aim we bound the associated Lebesgue function $\lambda$ which depends only on the points. We refer to this approach as {\em $\lambda$-greedy}.

\subsubsection{Lebesgue function and error estimation} We start by proving the following result. 

\begin{theorem}[Approximation error]
Let $f\in C([a, b])$ and let $f^\star_{X,a}\in {\mathcal E}_{X,a}$ be its best approximation in ${\mathcal E}_{X,a}$ with respect to the norm $\|\cdot\|_\infty$. Then it holds that
\begin{align}\label{eq:lebesgue_error}
\left|\left(f - I_{X,\alpha}(f)\right)(x)\right|\leq \left(1 + \lambda(x) \right) \left\|f - f^\star_{X,a}\right\|_\infty, \;\; x\in [a,b].
\end{align}
\end{theorem}

\begin{proof}
Since $I_{X,a}(g) = g$ for all $g\in {\mathcal E}_{X,a}$, and in particular for $g=f^\star_{X,a}$, we have that 
\begin{align}\label{eq:intermediate_lemm_proof}
\left|f(x) - I_{X,\alpha}(f)(x)\right|\nonumber
&=\left|f(x) - f^\star_{X,a}(x) + f^\star_{X,a}(x)- I_{X,\alpha}(f)(x)\right|\nonumber\\
&\leq\left|f(x) - f^\star_{X,a}(x)\right| + \left|f^\star_{X,a}(x)- I_{X,\alpha}(f)(x)\right|\nonumber\\
&=\left|f(x) -f^\star_{X,a}(x)\right| + \left|I_{X,\alpha}(f^\star_{X,a})(x)- I_{X,\alpha}(f)(x)\right|\nonumber\\
&=\left|f(x) - f^\star_{X,a}(x)\right| + \left|I_{X,\alpha}\left(f^\star_{X,a}- f\right)(x)\right|\nonumber\\
&\leq \left\|f - f^\star_{X,a}\right\|_\infty + \left|I_{X,\alpha}\left(f^\star_{X,a}- f\right)(x)\right|.
\end{align}
To bound the second term we use  \eqref{eq:lagrange_interpolant} and thus:
\begin{align*}
\left|I_{X,\alpha}\left(f^\star_{X,a}- f\right)(x)\right|
&=\left|\sum_{j=1}^n \left(f^\star_{X,a}- f\right)(x_j) \psi_j(x)\right|\\
&\leq \max\limits_{1\leq j \leq n} \left|\left(f^\star_{X,a}- f\right)(x_j)\right| \sum_{j=1}^n |\psi_j(x)|\\
&\leq \left\|f^\star_{X,a}- f\right\|_\infty  \lambda(x).
\end{align*}
Taking into account \eqref{eq:intermediate_lemm_proof}, we obtain
\begin{align*}
\left|f(x) - I_{X,\alpha}(f)(x)\right|\nonumber
&\leq \left\|f - f^\star_{X,a}\right\|_\infty +  \left\|f^\star_{X,a}- f\right\|_\infty \lambda(x)\\
& = \left\|f - f^\star_{X,a}\right\|_\infty \left( 1+ \lambda(x) \right),
\end{align*}
and this concludes the proof.
\end{proof}

As an illustrative example, in the same setting of Figure \ref{fig:1}, in Figure \ref{fig:2}, we plot the Lebesgue functions associated to $n=8$ equispaced, Halton and Chebyshev data and evaluated on $400$ equispaced points. Observe that in this case the Chebyshev points seem to not provide the smallest Lebesgue constant. This is in contrast with interpolation with global polynomials, but in agreement with other approximation methods.

\begin{figure}[!ht]
    \centering
   \includegraphics[clip,trim=2cm 0.73cm 2cm 0.73cm,scale = 0.35]{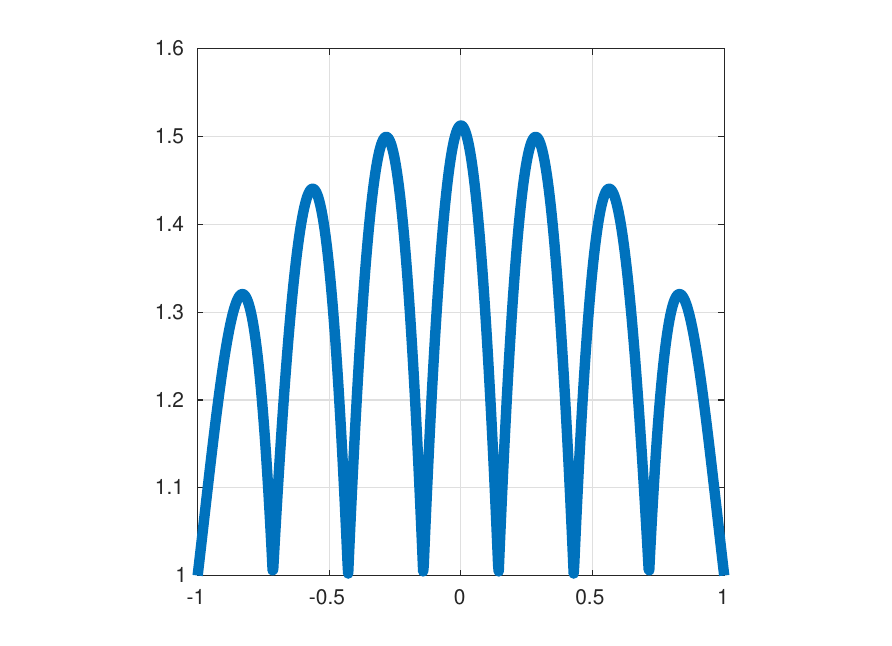}
    \includegraphics[clip,trim=2cm 0.73cm 2cm 0.73cm,scale = 0.35]{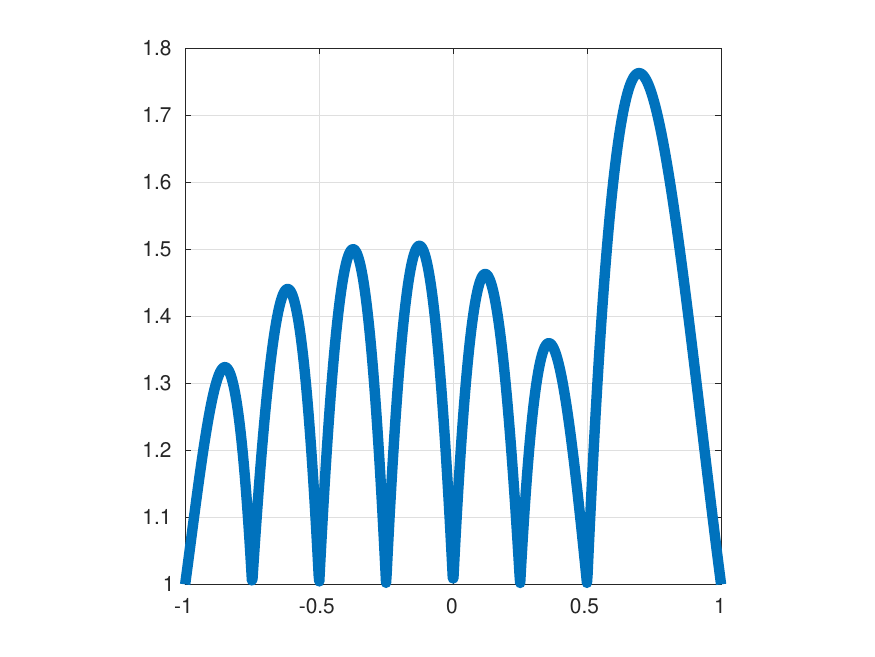}
    \includegraphics[clip,trim=2cm 0.73cm 2cm 0.73cm,scale = 0.35]{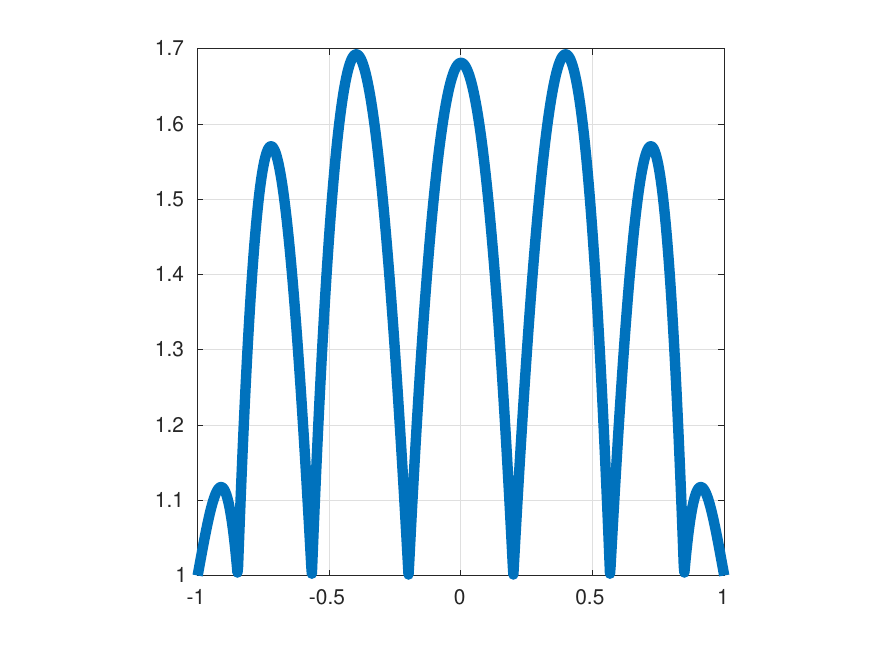}  
    \caption{From left to right: Lebesgue functions computed on $n=8$ equispaced, Halton and Chebyshev data, respectively.}
    \label{fig:2}
\end{figure}

\begin{remark}[Related results]
Observe that the error bound in \eqref{eq:lebesgue_error} is analogous but not equivalent to similar statements in other methods (e.g., polynomial or kernel-based interpolation). Indeed, the splitting of the error on the right hand side is only partially separating the $f$-dependent and the $f$-independent terms, since the best approximant $f_{X, a}^\star$ is depending on the interpolation points. In other words, one may try to minimize the first term to find \emph{good}, i.e., sub-optimal, interpolation points, but this may spoil the second term.
\end{remark}

\subsubsection{The $\lambda$-greedy algorithm}

Given such error bound, we introduce a new $\lambda$-greedy scheme that is defined as follows. Given $X$, $F$ and $\tau$, a fixed tolerance, the $\lambda$-greedy algorithm for exponential splines can be summarized Algorithm 3. 

\begin{algorithm}[H]
\small
\caption{\textbf{Pseudo-code for the $\lambda$-greedy algorithm}}
\begin{algorithmic}[1]	
\STATE Take an initial set of sorted data ${\tilde X}=\{x_i\}_{i=1}^q \subset X$; $x_1=a, x_q=b$ and $q\geq2$.
\STATE Compute $\lambda(x)$ with the initial set $\tilde{X}$.
\STATE While ${\rm argmax}_{x_i \in X\setminus {\tilde X}} \lambda(x) > \tau $:
    \begin{enumerate}    
    \item Define $x^*={\rm argmax}_{x_i \in X\setminus {\tilde X}} \lambda(x)$.
    \item Set ${\tilde X}= {\tilde X}\cup \{x^*\}$ and sort  ${\tilde X}$.
    \item Compute $\lambda(x)$ with the set $\tilde{X}$.
\end{enumerate}
 		\end{algorithmic}
\end{algorithm}

\begin{remark}[Computational aspects]
Observe that the efficient execution of the $\lambda$-greedy algorithm depends on the efficient computation of $\lambda$ and of $I_{\tilde X, \alpha}(f)$. Both of them can be be computed rather efficiently by means of the local basis. Indeed, in this case for all $x\in X$ one needs to locate the index $i$ such that $x\in [x_i, x_{i+1}]$, and then only perform local computations inside this interval.
\end{remark}

\begin{remark}
In the $\lambda$-greedy selection, we fix a tolerance for the Lebesgue functions. However, we are able to prove the efficacy, i.e. the convergence of the $\lambda$-greedy scheme, only numerically. As an illustrative example, in Figure \ref{fig:3}, we take $n=300$ equispaced nodes and we apply the $\lambda$-greedy scheme without any stopping rule, i.e. we extract ${\tilde n}=300$ nodes. This didactic example aims at understanding the behaviour of the Lebesgue constant when the number of nodes grows and how it relates with the conditioning of the problem. Precisely, from the first and second panel, we observe that  the Lebesgue constant initially decreases and then it saturates coherently with the condition number of the interpolation matrix. In the last panel we further show the sparsity of the collocation matrix that increases as the number of nodes increases. This empirically explains the behaviour of the condition number and Lebesgue functions. In other words, our $\lambda$-greedy is effective until both the condition number and the Lebesgue constant do not saturate. Then, as an alternative stopping rule, one may look at the difference between the Lebesgue constant or the condition number at two consecutive iterations of the $\lambda$-greedy scheme. 
\end{remark}

\begin{figure}[!ht]
    \centering
   \includegraphics[clip,trim=2cm 0.3cm 2cm 0.1cm,scale = 0.35]{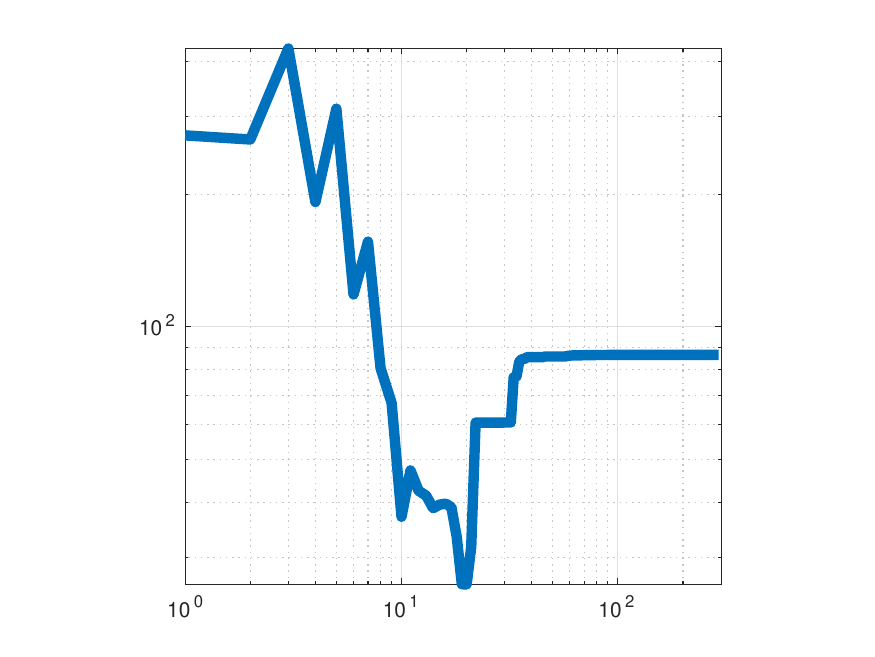}
    \includegraphics[clip,trim=2cm 0.3cm 2cm 0.1cm,scale = 0.35]{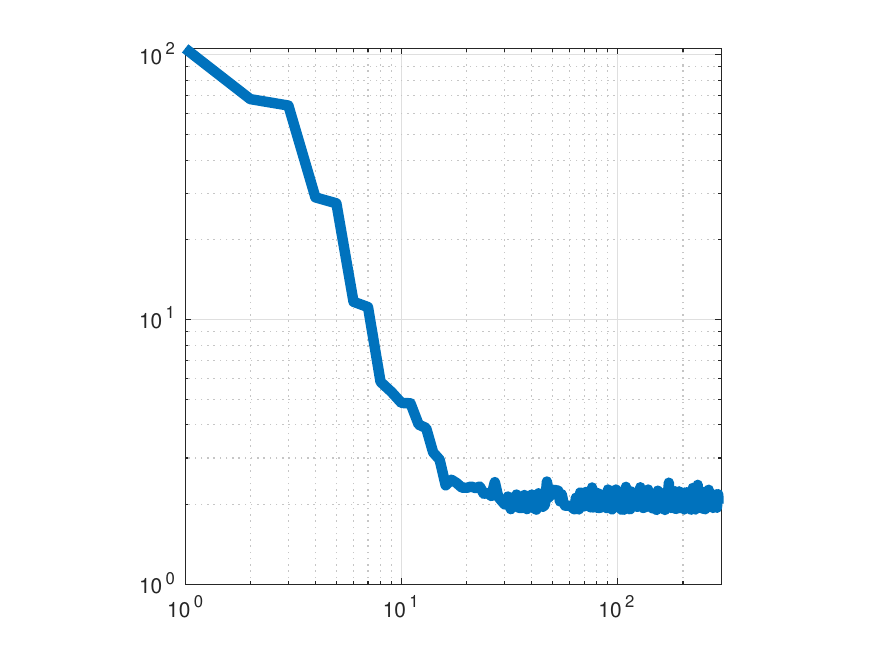}
    \includegraphics[clip,trim=2cm 0.3cm 2cm 0.1cm,scale = 0.35]{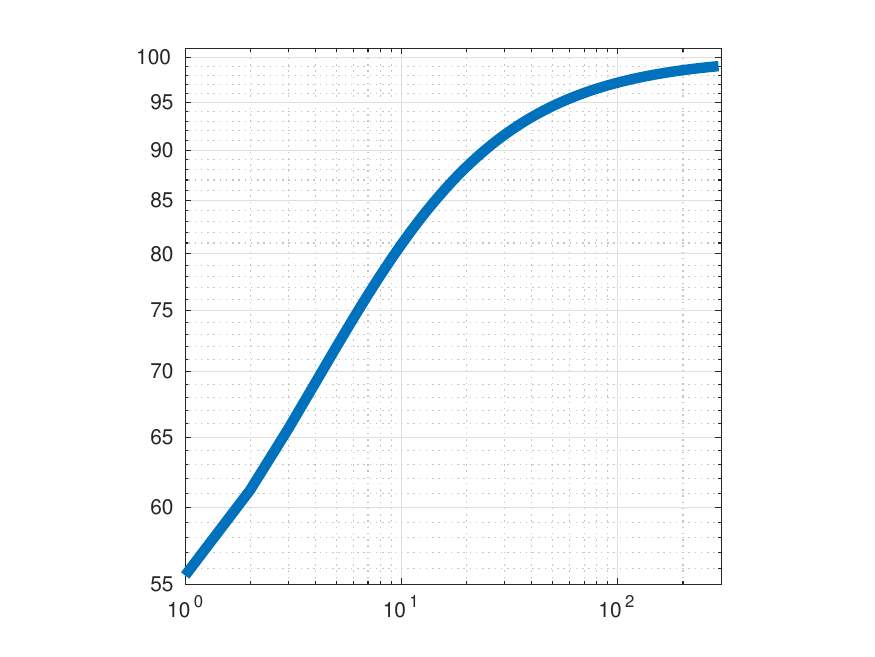}  
    \caption{Illustrative example of the $\lambda$-greedy extraction of $300$ nodes. At each step of the algorithm we compute the condition of the collocation matrix (left), the Lebesgue constant (middle) and the sparsity of the collocation matrix (right). Plots are in logarithmic scale.}
    \label{fig:3}
\end{figure}

\section{Numerical experiments}\label{sec:numerics}

In the following experiments, we test both the residual and the error-based schemes with different node distributions. Precisely, we consider equispaced data, Halton points and Chebyshev nodes. We further fix $\alpha = 2$.  Moreover, for all data sets, we take as initial set for the greedy strategy  {the first and last two nodes}. Tests have been carried out on a Intel(R) Core(TM) i7 CPU 4712MQ 2.13 GHz processor.

\subsection{Testing the $f$-greedy}

Throughout this subsection, we consider the following test function 
$$
f(x) = {\rm atan}(55x), \quad x \in [-1,1].
$$

As far as the $f$-greedy method which makes use of exponential basis functions is concerned, we fix the tolerance $\tau=10^{-3}$. In Figure \ref{fig:4}, we plot the results obtained by taking $300$ equispaced data, Halton points and Chebyshev nodes.  The number of extracted greedy nodes are respectively $\tilde{n}=36$, $30$ and $36$   that, as expected, cluster where the test function $f$ has steep gradients.  

\begin{figure}[!ht]
    \centering
    \hskip 0.1cm \includegraphics[clip,trim=2cm 0.73cm 2cm 0.1cm,scale = 0.35]{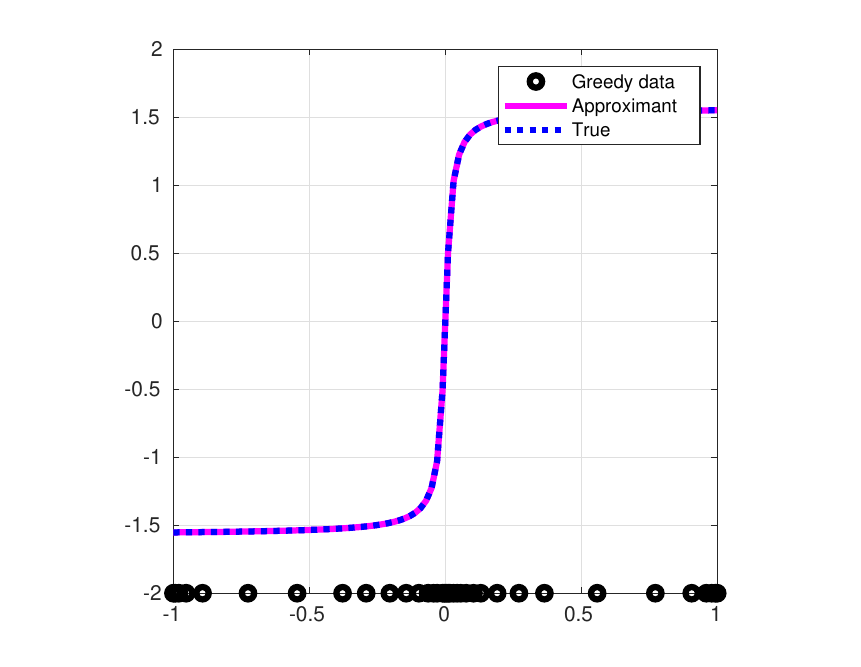}
    \includegraphics[clip,trim=2cm 0.73cm 2cm 0.1cm,scale = 0.35]{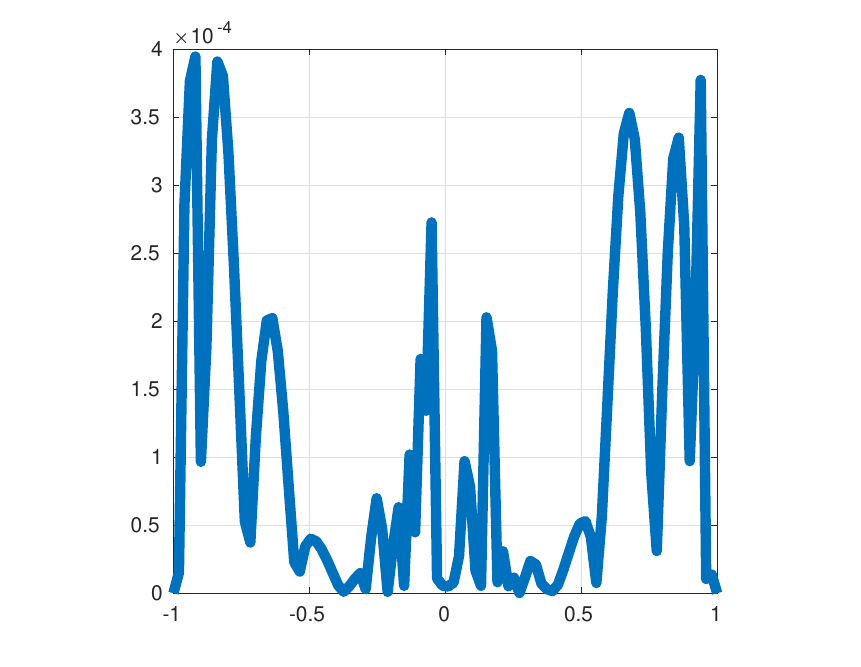}
    \includegraphics[clip,trim=2cm 0.73cm 2cm 0.1cm,scale = 0.36]{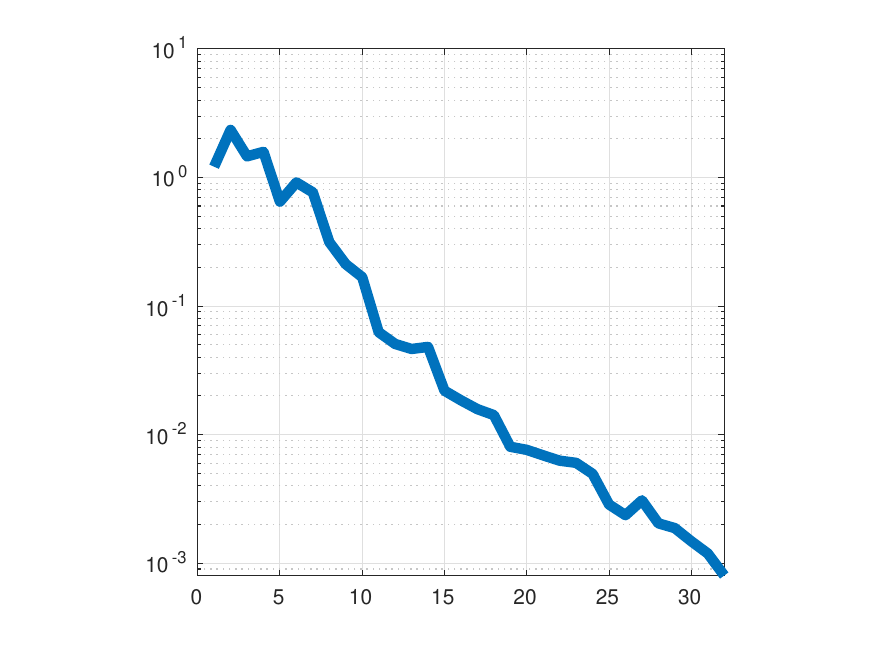}\\     \includegraphics[clip,trim=2cm 0.73cm 2cm 0.1cm,scale = 
0.35]{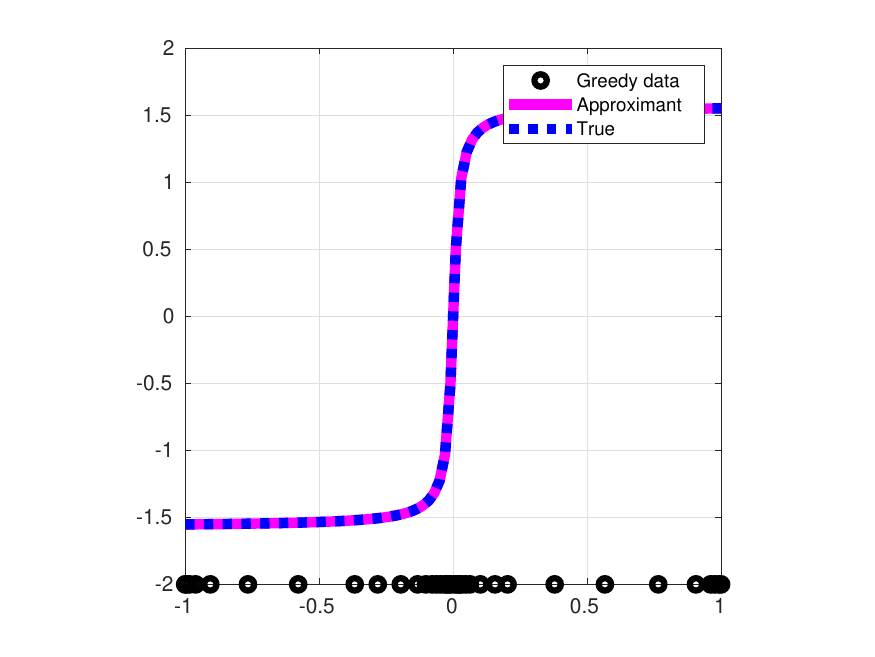}
    \includegraphics[clip,trim=2cm 0.73cm 2cm 0.1cm,scale = 0.35]{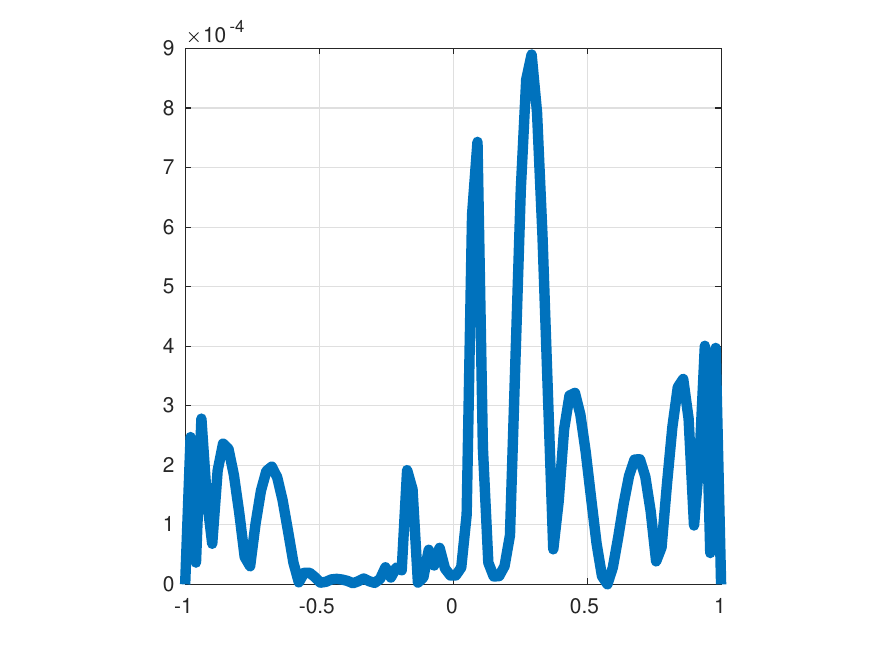}
    \includegraphics[clip,trim=2cm 0.73cm 2cm 0.1cm,scale = 0.35]{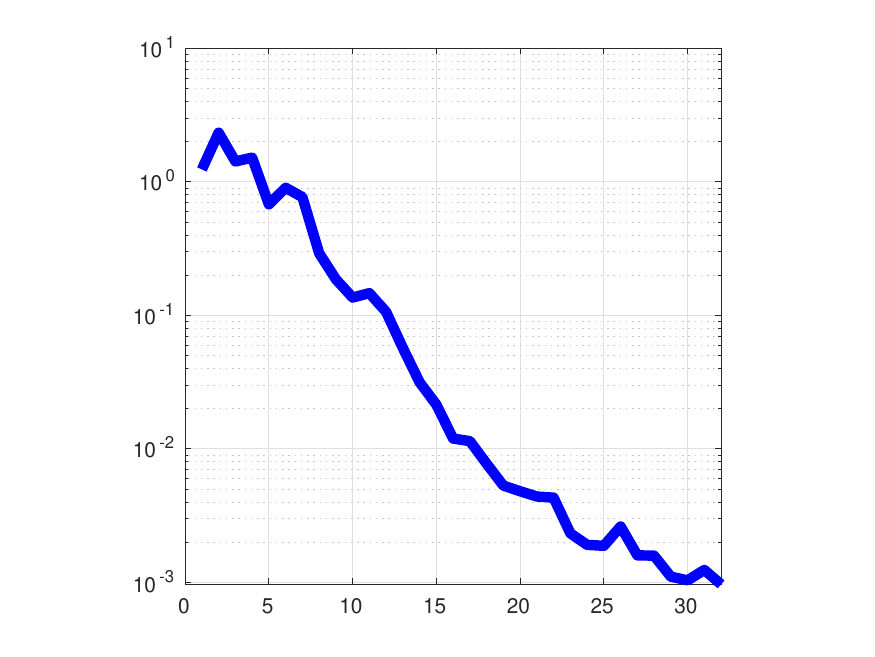}  \\ 
  \includegraphics[clip,trim=2cm 0.73cm 2cm 0.1cm,scale = 0.35]{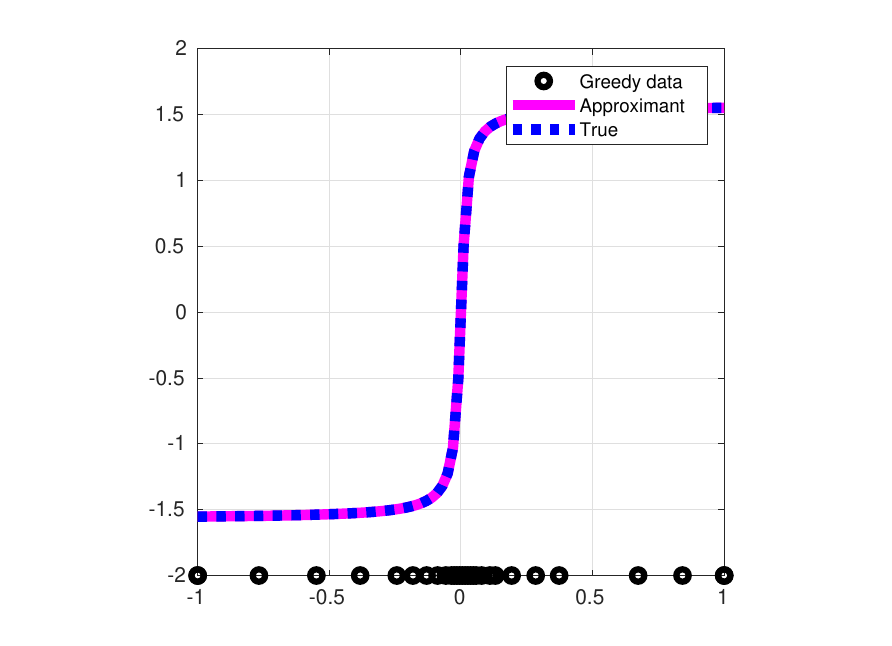}
    \includegraphics[clip,trim=2cm 0.73cm 2cm 0.1cm,scale = 0.35]{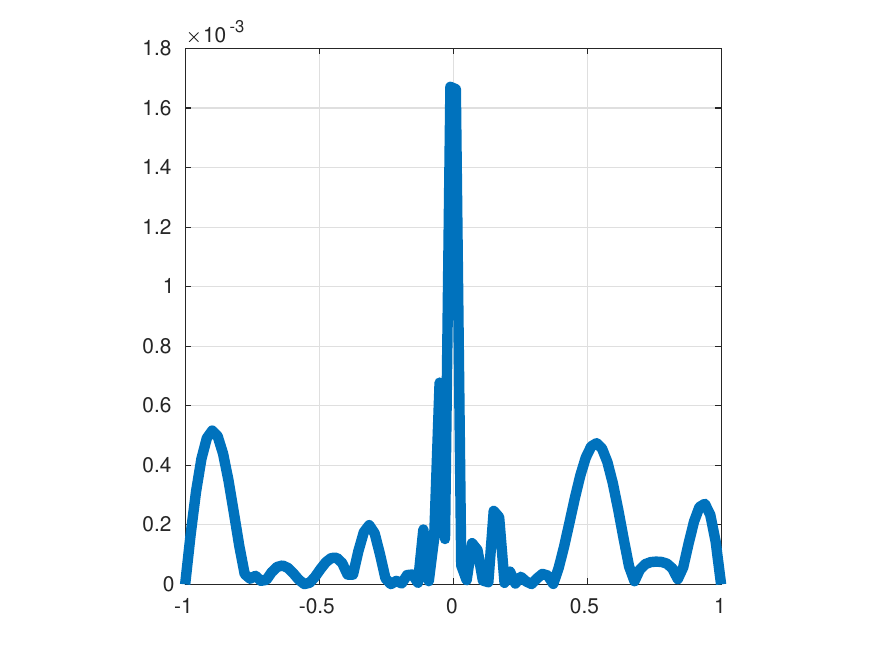}
    \includegraphics[clip,trim=2cm 0.73cm 2cm 0.1cm,scale = 0.35]{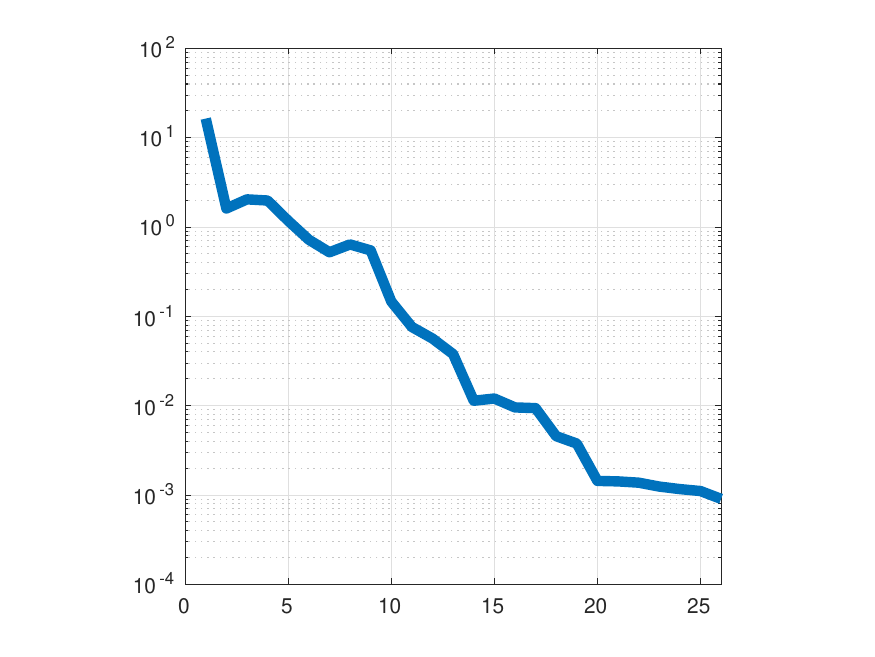}  
    \caption{Results for the $f$-greedy algorithm. First column: the extracted greedy data (black dots), the true function $f$ (blue dotted line) and the 
reconstructed function taking the greedy points (magenta solid line). Second column: the absolute error evaluated on $400$ equispaced data. Third column: the 
maximum of the residuals at each iteration of the greedy scheme. The experiment is carried out for equispaced, Halton and Chebyshev nodes, first, second and 
third row, respectively. }
    \label{fig:4}
\end{figure}

\subsection{Testing the $\lambda$-greedy}

One interesting feature of the $\lambda$-greedy scheme is that it is able to construct \emph{optimal} a priori node sets, provided that a sufficiently 
\emph{large} initial set of nodes is provided. To investigate the optimal data distribution for exponential splines, we take and initial set of $300$ equispaced 
data and we apply the $\lambda$-greedy scheme with $\tau = 2$. The result is depicted in Figure \ref{fig:5}, where we also show the same number of data (i.e. 
32) computed with the greedy algorithm and a kernel basis (thin plate splines). For kernels, as already known in literature, the points tend to distribute in a 
uniform way. When using exponential splines, the greedy data tend to cluster close to the boundary, showing some similarities with Chebyshev nodes that are 
known to be optimal for the monomial basis.     

\begin{figure}[!ht]
    \centering
    \includegraphics[clip,trim=1.5cm 5cm 0cm 5cm, scale=0.95]{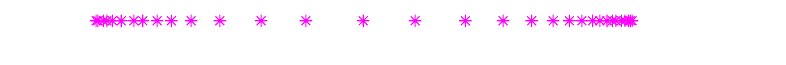}\\
    \includegraphics[clip,trim=1.5cm 5cm 0cm 5cm, scale=0.95]{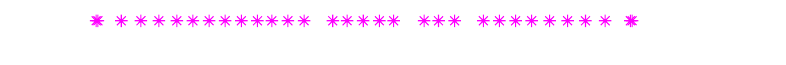}  
    \caption{Node distributions obtained via the residual-based greedy approach for EPS and kernels, respectively.}
    \label{fig:5}
\end{figure}

As last experiment, in Figure \ref{fig:6}, we plot the results of the $\lambda$-greedy scheme starting with $300$ equispaced, Halton and Chebyshev data. In this 
case, we fix the tolerance as $\tau = 3$. The algorithm    {selects} $\tilde{n}=18$, $19$ and $36$ equispaced, Halton and Chebyshev data, respectively. In all 
cases they cluster on the boundary. In the last column of Figure \ref{fig:6}, we report the Lebesgue constant at each iteration of the greedy scheme.  To get a 
feedback on the accuracy, with the reduced data, we reconstruct the function 
function $f(x)=x^2$. The associated absolute error is depicted in the second column of Figure \ref{fig:6}. 

\begin{figure}[!ht]
    \centering
   \includegraphics[clip,trim=2cm 0.73cm 2cm 0.1cm,scale = 0.35]{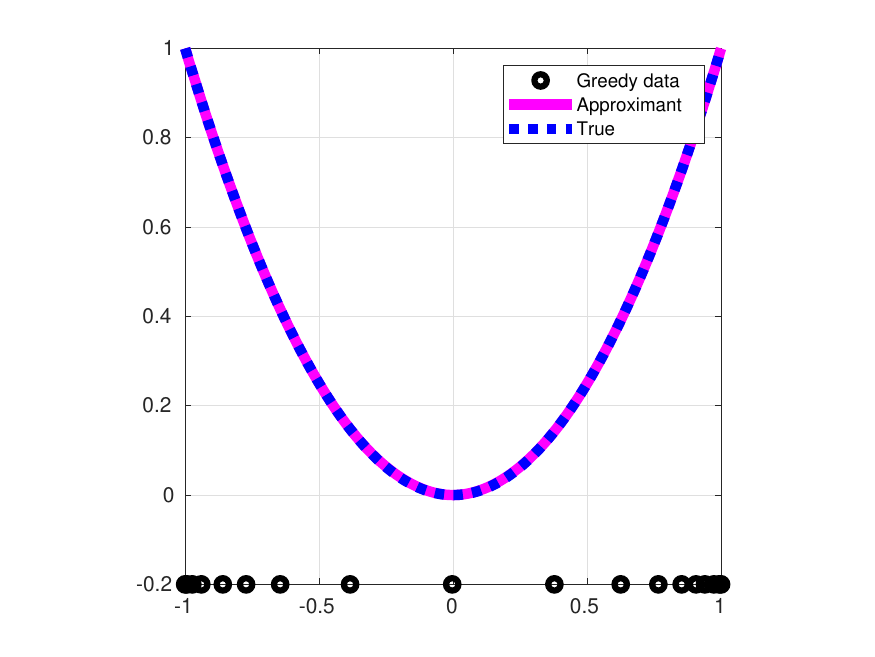}
    \includegraphics[clip,trim=2cm 0.73cm 2cm 0.1cm,scale = 0.35]{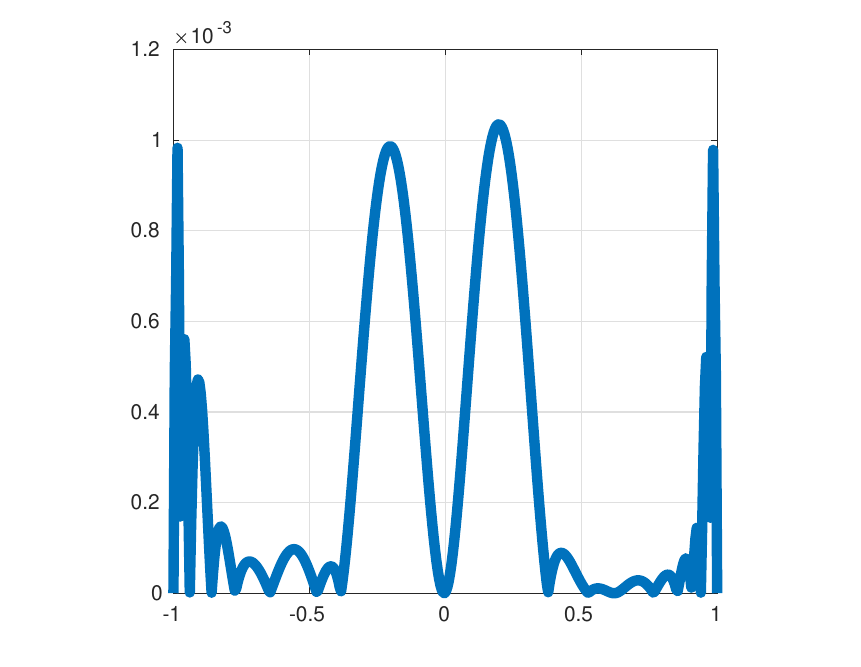}
    \includegraphics[clip,trim=2cm 0.73cm 2cm 0.1cm,scale = 0.35]{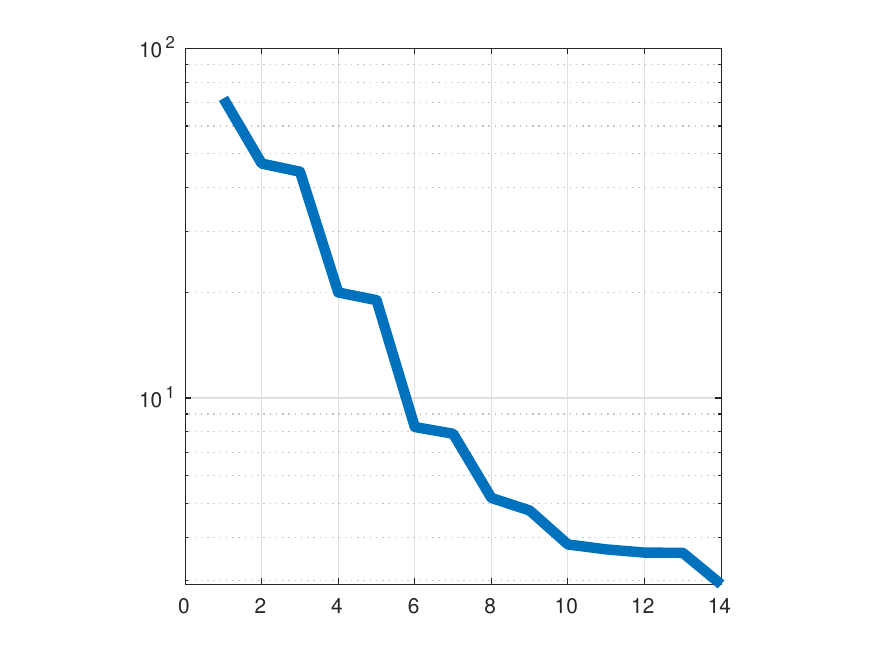}      
  \includegraphics[clip,trim=2cm 0.73cm 2cm 0.1cm,scale = 0.35]{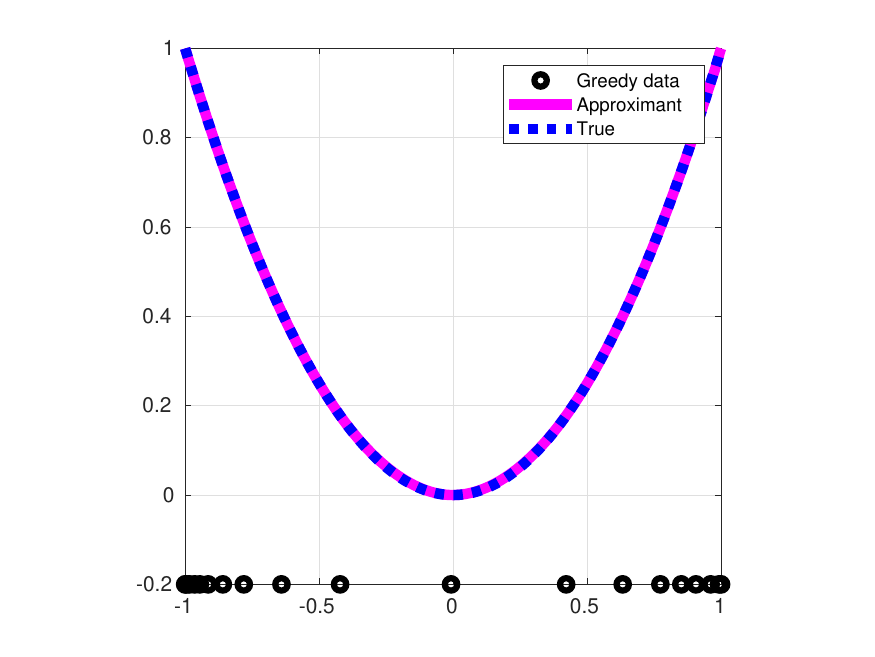}
    \includegraphics[clip,trim=2cm 0.73cm 2cm 0.1cm,scale = 0.35]{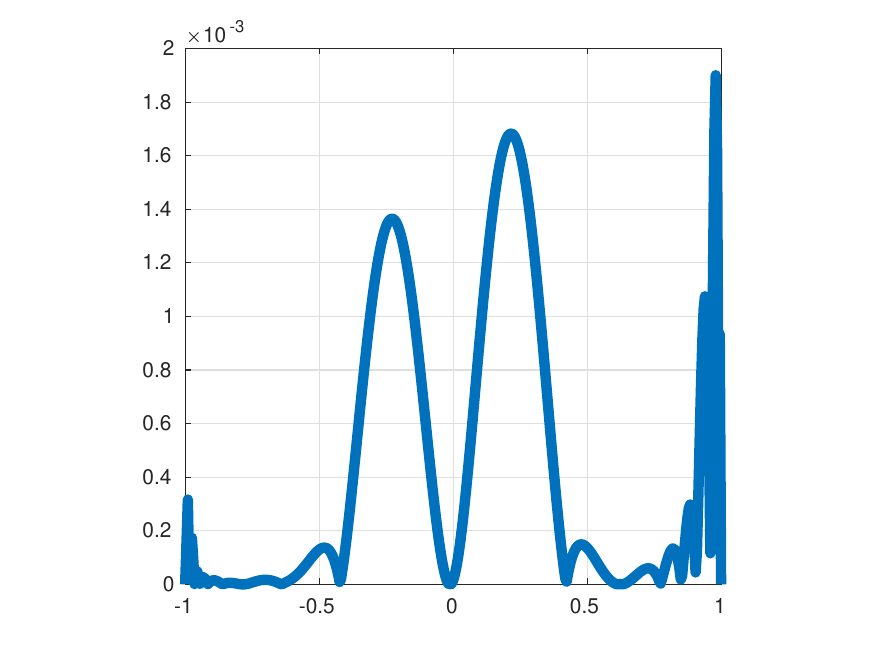}      
   \includegraphics[clip,trim=2cm 0.73cm 2cm 0.1cm,scale = 0.35]{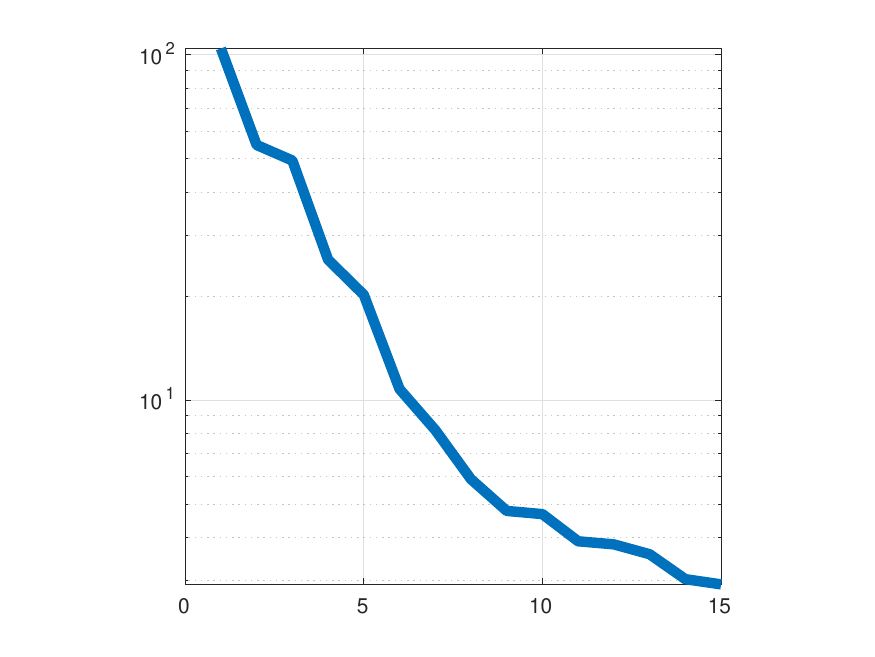}
    \includegraphics[clip,trim=2cm 0.73cm 2cm 0.1cm,scale = 0.35]{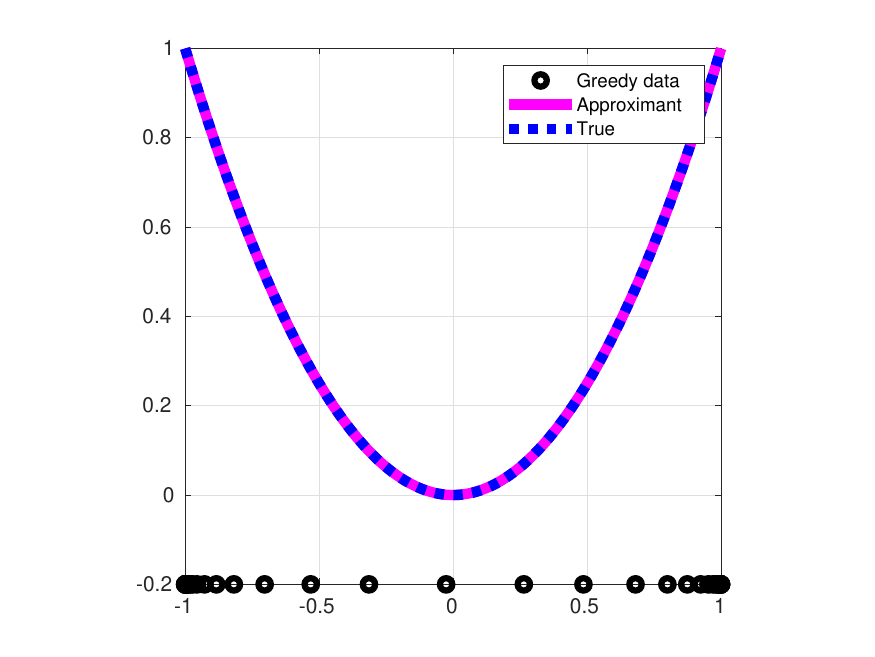}
    \includegraphics[clip,trim=2cm 0.73cm 2cm 0.1cm,scale = 0.35]{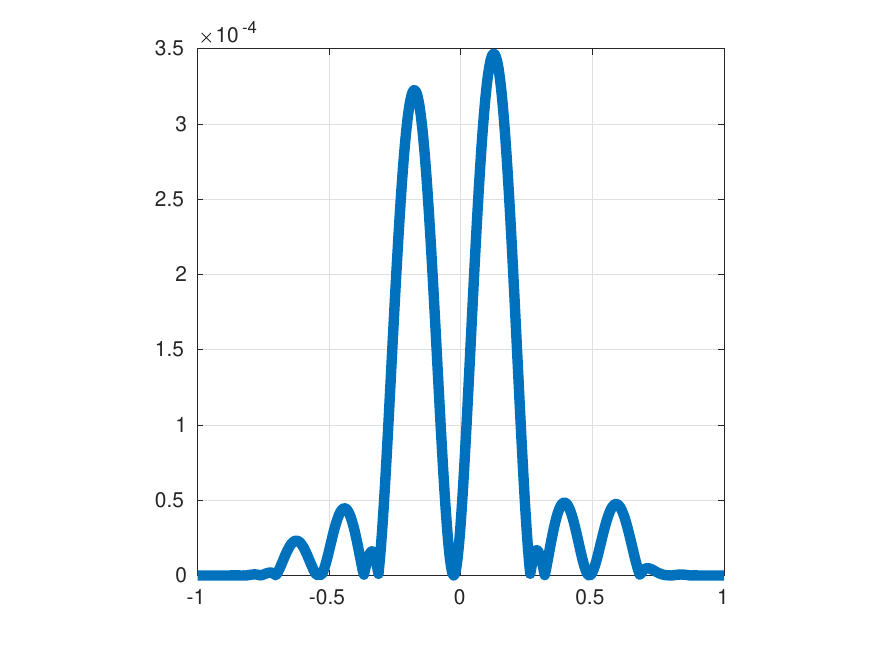}
    \includegraphics[clip,trim=2cm 0.73cm 2cm 0.1cm,scale = 0.35]{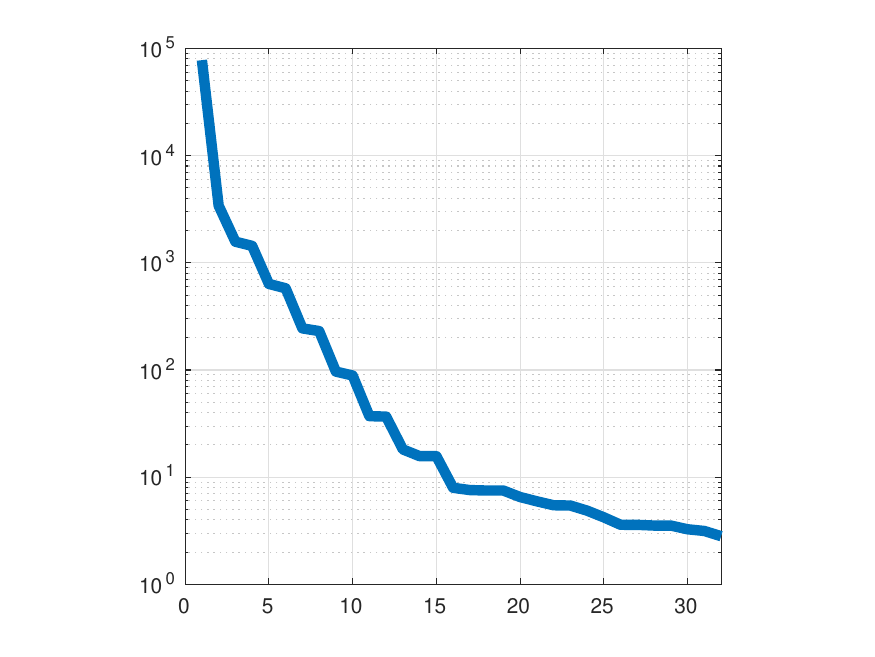}    
    \caption{Results for the $\lambda$-greedy algorithm. First column: the extracted greedy data (black dots), the true function $f$ (blue dotted line) and the reconstructed function taking the greedy points (magenta solid line). Second column: the absolute error evaluated on $400$ equispaced data. Third column: the Lebesgue constant at each iteration of the greedy scheme. The experiment is carried out for equispaced, Halton and Chebyshev nodes, first, second and third row, respectively.  }
    \label{fig:6}
\end{figure}

\section{Conclusions and work in progress}
\label{Concl}

We have investigated the use of greedy strategies for EPS interpolation. To this aim we have studied the cardinal form of the EPS interpolant and then we provided error bounds based on the Lebesgue functions. The results show that the error-based greedy points for EPS tend to cluster on the boundary of the approximation interval, despite the fact that Chebyshev points are not the optimal ones (this has been observed numerically via Figure \ref{fig:2}). 

Work in progress consists in investigating the proposed tool in applications, as in the context of Laplace transform inversion based on smoothing splines \cite{AMC_CCC}, as well as for interpolation/extrapolation algorithms for the inversion of the Fourier transform \cite{perracchione_2021}.

\section*{Acknowledgments} We thank  the support the GNCS-INdAM project \lq\lq Interpolazione e smoothing: aspetti teorici, computazionali e applicativi". This research has been done within the Italian Network on Approximation (RITA) and the thematic group on \lq\lq Approximation Theory and Applications" of the Italian Mathematical Union (UMI). EP acknowledges the financial contribution from the agreement ASI-INAF n.2018-16-HH.0.

\bibliography{biblio}
\bibliographystyle{abbrv}
\end{document}